\documentclass[12pt]{amsart}

\usepackage{color}
\usepackage{enumitem}

\usepackage{amsmath,amssymb,setspace,nicefrac, yhmath, amscd,eucal}
\usepackage[active]{srcltx}
\usepackage[pagebackref, colorlinks, linkcolor=red, citecolor=blue, urlcolor=blue, hypertexnames=true]{hyperref}
\usepackage{amsrefs}
\usepackage{tikz}

\allowdisplaybreaks
\usepackage[all]{xy}

\usepackage[active]{srcltx}

\newcommand{\R}{{\mathbb R}}

\newcommand{\Z}{{\mathbb Z}}
\newcommand{\N}{{\mathbb N}}

\newtheorem{thm}{Theorem}[section]

\newtheorem{lem}[thm]{Lemma}
\newtheorem{question}[thm]{Question}
\newtheorem{definition}[thm]{Definition}

\newtheorem{remark}[thm]{Remark}
\newtheorem{proposition}[thm]{Proposition}

\theoremstyle{definition}

\title[]{On continuous orbit equivalence rigidity for virtually cyclic group actions}
\author{Yongle Jiang}

\address{Y. Jiang, School of Mathematical Sciences, Dalian University of Technology, Dalian, 116024, China}

\email{yonglejiang@dlut.edu.cn}

\begin{document}
\date{\today}

\begin{abstract}
We prove that for any two continuous minimal (topologically free) actions of the infinite dihedral group on an infinite compact Hausdorff space, they are continuously orbit equivalent only if they are conjugate. We also show the above fails if we replace the infinite dihedral group with certain other virtually cyclic groups, e.g. the direct product of the integer group with any non-abelian finite simple group.
\end{abstract}



\maketitle

\section{introduction}

Let $G$ be a countable discrete group and $X$ be a compact Hausdorff space. Denote by $\mathcal{C}$ a class of continuous actions of $G$ on $X$. A natural and classical topic is  to classify members in $\mathcal{C}$ up to certain equivalence relation. In classical topological dynamical system, people have studied intensively the classification of $\Z$-actions, or amenable group actions in general, up to conjugacy. Historically speaking, the notion of entropy was invented to distinguish Bernoulli shifts on finite-bases with different sizes up to conjugacy (first in the measurable setting). Since one decade ago, more and more attention has been put on a much wider class of acting groups, i.e. sofic groups after the pioneering work of L. Bowen on sofic entropy in the measurable setting \cite{Bow} and D. Kerr and H. Li shortly in the topological setting \cite{KL1}.

In contrast to the conjugacy relation, X. Li initiated a systematic study of the notion of continuous orbit equivalence (see Definition \ref{def: coe}) for topological free actions of a general group $G$ \cite{Li} during his study of crossed product C$^*$-algebras. Note that this notion and its weaker versions, e.g. topological orbit equivalence have been studied in special cases (see \cite{BT, GPS95, Su}) before Li's work. In particular, Giordano-Putnam-Skau \cite{GPS95} and Giordano-Matui-Putnam-Skau \cite{GMPS08, GMPS10} have proved a series of remarkable results on classification of minimal $\Z$-actions on the Cantor set up to topological orbit equivalence.

Although continuous orbit equivalence relation is by definition, a priori, weaker than the conjugacy relation, an interesting phenomenon is that these two notions may coincide in certain cases.
To better describe this, let us say a continuous action $\alpha$ of $G$ on $X$ is  \emph{continuous orbit equivalence rigid for $\mathcal{C}$} if for any continuous action $\beta\in \mathcal{C}$, $\alpha$ is continuously orbit equivalent to $\beta$ implies that they are conjugate, denoted by $\alpha\overset{coe}{\sim}\beta\Rightarrow \alpha\overset{conj.}{\sim}\beta$. Note that here, compared with X. Li's definition \cite[Definition 2.5]{Li}, we further assume the acting groups are the same.

Now, we take $\mathcal{C}$ to be the class of minimal topologically free actions
and explain two motivations behind this work.

First, continuous orbit equivalence rigidity has been established for various group actions, see e.g. \cite{Li, BT, Sc1, Sc2, CM, CJ}. However, as far as we know, the acting groups are assumed to be torsion-free in all written proofs (see also Remark \ref{remark: odometer actions are done in CM} and Remark \ref{remark: torsion free assumption is unnecessary for full shifts}). Thus it is natural to ask whether continuous orbit equivalence rigidity can be established when the acting groups contain non-trivial torsion elements. Second, a well-known result due to Boyle-Tomiyama \cite{BT} says that two topologically transitive, topologically free $\Z$-actions (e.g. minimal free $\Z$-actions) are continuously orbit equivalent only if they are conjugate. This result has been extended to minimal equicontinuous actions of $\Z^{d}$ ($d\geq 2$) in \cite{CM}. More precisely, it is shown there that for minimal equicontinuous $\Z^d$-systems, continuous orbit equivalence implies that the systems are virtually piecewise conjugate. Moreover, actions of virtually cyclic groups have been considered in the study of dynamic asymptotic dimension \cite{ALSS} and Matui's HK conjecture \cite{OS, Sca} recently. Thus, it is natural to study the analogy of Boyle-Tomiyama's above mentioned result for virtually cyclic group actions, i.e. we ask the following question.
\begin{question}
Can we establish continuous orbit equivalence rigidity for minimal topologically free actions of infinite virtually cyclic (but non-cyclic) groups?
\end{question}

In general, we could not hope continuous orbit equivalence rigidity holds true for all infinite virtually cyclic groups. Indeed, we have the following theorem.

\begin{thm}\label{thm: - result}
Let $F$ be any non-trivial finite group with trivial center, e.g. a non-abelian finite simple group. Then there exist two minimal free actions $\widetilde{\alpha}$ and $\widetilde{\alpha}'$ of $\Z\times F$ on an infinite compact Hausdorff space $X$ such that $\widetilde{\alpha}\overset{coe}{\sim}\widetilde{\alpha}'$ but $\widetilde{\alpha}\overset{conj.}{\not\sim}\widetilde{\alpha}'$.
\end{thm}
On the positive side, when dealing with the infinite dihedral group, we do have such a rigidity theorem.
\begin{thm}\label{thm: + result}
Let $\alpha$ and $\beta$ be any two minimal (topologically free) actions of the infinite dihedral group $D_{\infty}=\Z\rtimes \Z/{2\Z}$ on an infinite compact Hausdorff space $X$. Then $\alpha\overset{coe}{\sim}\beta\Rightarrow\alpha\overset{conj.}{\sim}\beta$. 
\end{thm}
Note that minimal actions of $D_{\infty}$ on an infinite compact Hausdorff space is automatically topologically free, see Proposition \ref{prop: minimal implies topologically free}. Moreover,
under the assumptions in the above theorem, we know continuous orbit equivalence implies the two actions have the same topological entropy. This may be compared with \cite[Theorem 7.5]{KT} and \cite[Theorem D]{KL2}, where it is shown that many non-virtually cyclic groups admit actions for which topological entropy is an invariant of continuous orbit equivalence.
For possible generalizations of the theorems to other virtually cyclic groups, see the discussion in Section 
\ref{sec: remarks}. Here, we content ourselves with not discussing a general version of the above theorem. Below, we briefly describe strategy for the proofs.

For the proof of Theorem \ref{thm: - result}, we use the skew product construction to reduce it to a problem of constructing minimal actions admitting continuous cocycles into finite groups which are not coboundaries. On the positive side, one may compare the proof of Theorem \ref{thm: + result} with the proof of Boyle-Tomiyama \cite{BT} for continuous orbit equivalence rigidity for minimal $\Z$-actions. Here, the basic strategy is as follows:  by a criterion \cite[Proposition 4.4]{Li} due to X. Li, it suffices to show the orbit cocycle $c$ is cohomologous to a group isomorphism of $D_{\infty}$. 
The starting point for our proof is the characterization of bi-Lipschitz bijections of $\Z$ due to Benjamini-Shamov \cite{BS}.
After restricting the orbit cocycle $c:\Z\times X\to D_{\infty}$ on the sub $\Z$-action, we may extract from it a defect cocycle $a$ which attains only finitely many values in $D_{\infty}$. Here, $a$ is defined to measure the defect between $c$ and a group homomorphism from $\Z$ to $D_{\infty}$. By twisting $a$ with a suitable coboundary, we can assume the range of $a$ is contained in $\Z$. If the sub $\Z$-action is minimal, one can apply the well-known Gottschalk-Hedlund theorem \cite{GH} for minimal $\Z$-actions to deduce $a$ is actually a coboudary, which implies the original orbit cocycle $c$ is trivial. For the general case, i.e. the sub $\Z$-action is not minimal, we can apply Gottschalk-Hedlund theorem to each $\Z$-component, now the existence of reflection of $\Z$ in the group structure of $D_{\infty}$ allows us to patch up untwisters for each component to get a global one.

The paper is organized into five sections.  In Section \ref{section: preliminaries}, we explain terminologies used in this paper, including cocycles, continuous orbit equivalence, skew product actions and induced actions. We also present examples of continuous actions. Then we give examples of minimal free actions admitting continuous cocycles which are not coboundaries in Section \ref{section: existence of non-coboundaries}. The proof of Theorem \ref{thm: - result} (resp. Theorem \ref{thm: + result}) is given in Subsection \ref{subsection:-part} (resp. Subsection \ref{subsection: +part}). We conclude the paper with several remarks on the possible generalization of the theorems in Section \ref{sec: remarks}. 

In this paper, $G$ and $H$ usually mean virtually cyclic groups or finitely generated groups, which is clear from the context. Moreover, we denote a target group by $K$. Unless otherwise specified, all cocycles are assumed to be continuous.

\section{preliminaries}\label{section: preliminaries}

\subsection{Cocycles and continuous orbit equivalence}
Let $G\curvearrowright X$ be a continuous action on a compact Hausdorff space and let $K$ be a discrete group. Recall that a continuous map $c: G\times X\to K$ is called a \emph{$K$-valued (continuous) cocycle} if $c(g_1g_2, x)=c(g_1, g_2x)c(g_2, x)$ for all $g_1$, $g_2\in G$ and all $x\in X$. Two continuous cocycles $c_1, c_2: G\times X\to K$ are  \emph{(continuously) cohomologous/equivalent} if there exists a continuous map $b: X\to K$ such that $c_1(g, x)=b(gx)^{-1}c_2(g, x)b(x)$ for all $g\in G$ and all $x\in X$. A continuous cocycle is called \emph{trivial (resp. a cobundary)} if it is equivalent to a group homomorphism $\phi: G\to K$ (resp. the trivial group homomorphism $\phi\equiv e_K$ on $G$), which is treated as a constant cocycle on $X$.

Note that when $G=\Z$, a $K$-valued cocycle $c$ for $\Z\curvearrowright X$ can be represented as a map $f: X\to K$. Indeed, for one direction, set $f(x)=c(1, x)$; for the other direction, we may define $c(1, x)=f(x)$ and use the cocycle identity to define $c(N, x)$ for all $N\in \Z$ and hence get a cocycle $c$. Sometimes, we may abuse our notation by calling $f$ a cocycle.

Given two continuous actions $G\curvearrowright X$ and $K\curvearrowright Y$, we have the notion of continuous orbit equivalence (written as coe for short) which is systematically studied by X. Li \cite[Definition 2.5]{Li}. 

\begin{definition}[coe]\label{def: coe}
$G\curvearrowright X$ and $K\curvearrowright Y$ are continuously orbit equivalent, denoted by  $G\curvearrowright X\overset{coe}{\sim}K\curvearrowright Y$, if there exists a homeomorphism $\phi: X\approx Y$ with inverse $\psi=\phi^{-1}$ and continuous maps $a: G\times X\to K$ and $b: K\times Y\to G$ such that 
\begin{align*}
\phi(gx)=a(g, x)\phi(x),~ \psi(hy)=b(h, y)\psi(y)
\end{align*}
for all $g\in G$, $h\in H$, $x\in X$ and $y\in Y$.
\end{definition}
If both $a$ and $b$ are further assumed to be group isomorphisms from $G$ to $K$ and from $K$ to $G$ respectively, then we call these two actions are \emph{conjugate}, written as $G\curvearrowright X\overset{conj.}{\sim}K\curvearrowright Y$. Clearly, conjugacy implies coe.

We remark that if $K\curvearrowright Y$ is \emph{topologically free} (i.e. for each $e\neq h\in K$, $\{y\in Y: hy\neq y\}$ is dense in $Y$), then $a$ is a continuous cocycle, see \cite[Lemma 2.8]{Li}. When both $a$ and $b$ are cocycles in the above definition, we simply call $b$ the inverse cocycle of $a$.

The following fact is proved independently in \cite{MST} and \cite{Li_qi}, we sketch the proof for completeness.

\begin{proposition}\label{prop: from coe to bi-Lipschitz maps}
Let $G$ and $K$ be finitely generated groups. Let $G\curvearrowright X$ and $K\curvearrowright Y$ be two topological free actions on compact spaces. Assume they are continuously orbit equivalent and $a: G\times X\to K$ is the associated continuous orbit cocycle, then $G\ni g\mapsto a(g,x)\in K$ is a bi-Lipschitz bijection for each $x\in X$ with respect to the right invariant word metrics on $G$ and $K$ respectively. 
\end{proposition}
Throughout the paper, a \emph{bi-Lipschitz bijection} means a bijective map $\phi$ such that both $\phi$ and $\phi^{-1}$ are Lipschitz maps in the usual sense.
\begin{proof}
Let $S$ and $T$ be any symmetric generating set for $G$ and $K$ respectively. Let $|\cdot|_S$ and $|\cdot|_T$ be associated word lengths, and let $d_S(\cdot, \cdot)$ and $d_T(\cdot, \cdot)$ be the corresponding right invariant word length metrics on $G$ and $K$, i.e. $d_S(g_1, g_2)=|g_2g_1^{-1}|_S$ and $d_T(h_1, h_2)=|h_2h_1^{-1}|_T$ for all $g_1$, $g_2\in G$ and $h_1$, $h_2\in K$.

Then $d_T(a(g_1, x), a(g_2, x))=|a(g_2, x)a(g_1, x)^{-1}|_T=|a(g_2g_1^{-1}, g_1x)|_T\leq |g_2g_1^{-1}|_S\cdot \sup_{x\in X}\max_{s\in S}|a(s, x)|_T=d_S(g_1, g_2)\cdot \sup_{x\in X}\max_{s\in S}|a(s, x)|_T$, where the inequality follows from the cocycle identity. This shows $g\mapsto a(g, x)$ is a Lipschitz map. 

To see $g\mapsto a(g, x)$ is a bijection and its inverse is a Lipschitz map, apply the similar argument to the inverse cocycle using \cite[Lemma 2.10]{Li}.
\end{proof}

The following lemma is well-known, but we present its proof for the reader's convenience.

\begin{lem}\label{lem: independent of target group for being coboundaries}
Let $\alpha: \Z\curvearrowright X$ be a minimal action on a compact space $X$ and $K_0\subset K$ be a closed subgroup in a topological group. Let $c: \Z\times X\to K_0$ be a continuous cocycle such that $c(n, x)=f(\alpha_n(x))f(x)^{-1}$ for some continuous map $f: X\to K$ and all $n\in \Z$ and $x\in X$. Then there is a continuous map $f': X\to K_0$ such that $c(n, x)=f'(\alpha_n(x))f'(x)^{-1}$ holds for all $n\in \Z$ and all $x\in X$.
\end{lem}
\begin{proof}
Observe that the continuous map 
\[x\in X\to K_0f(x)\in K_0\backslash K\]
is $\alpha$-invariant. Hence minimality implies there exists some $k\in K$ such that $K_0f(x)=K_0k$ for all $x\in X$, i.e. $f(x)\in K_0k$ for all $x\in X$. Define $f'(x)=f(x)k^{-1}$. We still have $c(n, x)=f'(\alpha_n(x))f'(x)^{-1}$.
\end{proof}

\subsection{Skew product actions}\label{subsection: definition of skew product actions}
We recall the following well-known construction of skew product actions.
Let $\sigma: H\to Aut(K)$ be a group homomorphism. Let $c: H\curvearrowright X\to K$ be a \emph{skew cocycle}, i.e. for all $h_1$, $h_2\in H$ and all $x\in X$, 
\[c(h_1h_2, x)=c(h_1, h_2x)\sigma_{h_1}(c(h_2, x)).\]
One may check that $c((h_1h_2)h_3, x)=c(h_1(h_2h_3), x)$ holds for all $h_i\in H$, $1\leq i\leq 3$ and all $x\in X$. Besides, $c$ is a skew cocycle iff the map $c': H\curvearrowright X\to K\rtimes_{\sigma}H$ given by $c'(h, x)=(c(h, x), h)$ is a cocycle in the usual sense.

Given a group homomorphism $\sigma: H\to Aut(K)$, an action $H\curvearrowright X$ and a skew cocycle $c$ defined as above, we may define the \emph{generalized skew product action}:
 \begin{align*}
G:=K\rtimes_{\sigma} H&\curvearrowright X\times K\\
(k, h)(x, k')&:=(hx, c(h, x)\sigma_h(k')k^{-1}).
  \end{align*}
Note that when $\sigma$ is trivial, i.e. $\sigma\equiv id_K$, skew cocycles (resp. skew product actions) are reduced to the usual cocycles (resp. usual skew product actions) for the group $K\times H$. In this paper, we usually take $H=\Z$, $K$ be a finite group and $\sigma$ be trivial.
To emphasize the cocycle $c$, we usually denote $X\times K$ by $X\times_c K$.

\subsection{Induced actions}

Let $H$ be a subgroup of $G$. Fix any lift map $L: G/H\to G$ such that $L(gH)H=gH$ for all $g\in G$. Consider the associated cocycle $c: G\times G/H\to H$ given by $c(g, g'H)=L(gg'H)^{-1}gL(g'H)$ for all $g, g'\in G$. Given any continuous action $H\curvearrowright Y$, then the \emph{induced action} $G\curvearrowright G/H\times Y$ is defined as follows:
\[g(g'H, y)=(gg'H, c(g, g'H)y),~\mbox{for all $g, g'\in G$ and $y\in Y$.}\] 


\subsection{Examples of continuous actions}

\paragraph*{\textbf{Odometer actions}}
Fix a sequence of strictly increasing positive integers $(n_i)_{i\geq 1}$ such that $n_i\mid n_{i+1}$ for all $i\geq 1$. Let $\lim\limits_{\leftarrow}\Z/{n_i\Z}=\{(x_i+n_i\Z)_{i\geq 1}\in \prod_{i\geq 1}\Z/{n_i\Z}:~n_i\mid (x_{i+1}-x_i),~\forall~i\geq 1\}$.  The \emph{odometer action associated to $(n_i)$} is defined as the continuous action $\alpha:\Z\curvearrowright \lim\limits_{\leftarrow}\Z/{n_i\Z}$ given by \[\alpha_n((x_i+n_i\Z)_i)=(n+x_i+n_i\Z)_i,\] where $x_i\in \Z$.

\paragraph*{\textbf{Weakly mixing actions}}
A continuous action $\Z\curvearrowright X$ on a compact Hausdorff space is called \emph{(topologically) weakly mixing} \cite[p.23]{Gl} if the product action $\Z\curvearrowright X\times X$ is \emph{transitive}, i.e. for every pair of non-empty open sets $U$, $V$ in $X\times X$ there exists some $g\in \Z$ with $gU\cap V\neq\emptyset$. Examples of free minimal weakly mixing actions include the so-called Chac\'on system \cite[p. 27]{Gl}.

It is well-known and easy to check that for a continuous action $\Z\curvearrowright X$ on a compact Hausdorff space $X$, if $X$ has no isolated points and there is a point in $X$ with dense orbit, then $\Z\curvearrowright X$ is transitive. Conversely, for a compact metric space $X$, $\Z\curvearrowright X$ is transitive implies it admits a point with dense orbit, see e.g. \cite[Proposition 7.9]{KL_book}.

\subsection{Subgroups of $D_{\infty}$}

Let $D_{\infty}=\Z\rtimes \frac{\Z}{2\Z}=\langle s\rangle\rtimes \langle t\rangle=\langle s, t~|~t^2, tsts\rangle$ be the infinite dihedral group. We record a simple lemma on the structure of its subgroups.
\begin{lem}\label{lem: subgroups of the infinite dihedral group}
If $H$ is a subgroup of $D_{\infty}$, then either $H=\langle s^k, s^it\rangle$ for some $k\in\Z$ and some $0\leq |i|<k$, or $H=\langle s^kt\rangle$ or $H=\langle  s^k\rangle$  for some $k\in \Z$. Moreover, all subgroups can be realized as $\phi_i(H)$ for some $i\in\Z$ and some $H\in \{k\Z\rtimes \frac{\Z}{2\Z}, ~k\Z,~\frac{\Z}{2\Z}:~k\in\Z\}$, where, $\phi_i\in Aut(G)$ is defined by $\phi_i(s)=s, \phi_i(t)=s^it$.
\end{lem}
\begin{proof}
Write $H\cap \Z=\langle s^k\rangle$ for some $k\in \Z$. Note that the set of torsion elements of $G$ is precisely $\{s^it: i\in \Z\}$.

If $H$ contains no non-trivial torsion elements, then $H\subset \Z$, i.e. $H=H\cap \Z=\langle s^k\rangle$.

If $H$ contains some non-trivial torsion element, say $s^it\in H$ for some $i\in \Z$. W.L.O.G., we may assume $|i|$ is smallest among all non-trivial torsion elements in $H$. 

Case 1: $i=0$. Then clearly, $H=\langle s^k, t\rangle$.

Case 2: $|i|>0$. 

If $k=0$, then $H=\langle s^it\rangle$. Indeed, assume not, then $s^jt\in H$ for some $j\neq i$, which implies that $(s^it)(s^jt)=s^{i-j}\in H\cap \Z$, a contradiction.

We may assume $k>0$.  Clearly, $0<|i|<k$, then $H=\langle s^k, s^it\rangle$. Indeed, assume $s^jt\in H$, then $(s^jt)(s^it)^{-1}=s^{j-i}\in H\cap \Z=\langle s^k\rangle$, hence $s^jt\in \langle s^k, s^it\rangle$.

The moreover part is now clear.
\end{proof}
\begin{remark}\label{remark: odometer actions are done in CM}
Note that this lemma implies that for two odometer actions of $D_{\infty}$, the notion of structurally conjugacy between them, as introduced in \cite[Definition 3.1]{CM},   coincides with the usual conjugacy. Thus \cite[Theorem 3.3]{CM} implies that for two free odometer actions of $D_{\infty}$, continuous orbit equivalence between them implies they are conjugate.
\end{remark}
Next, let us observe the following hold, see also \cite[Proposition 2.8]{OS}.

\begin{proposition}\label{prop: minimal implies topologically free}
If $\alpha: D_{\infty}\curvearrowright X$ is a minimal action on an infinite compact Hausdorff space $X$, then it is topologically free.
\end{proposition}
\begin{proof}
Without loss of generality, we may assume $\alpha$ is not free. Write $D_{\infty}=\Z\rtimes \frac{\Z}{2\Z}=\langle s, t|~t^2, tsts\rangle=\langle s\rangle\rtimes \langle t\rangle$. Then observe that the sub $\Z=\langle s\rangle$-action, i.e. $\alpha|_{\Z}: \Z\curvearrowright X$ is minimal.

Indeed, note that the number of minimal $\Z$-components in $X$ is either one or two (See the discussion before Case I in the proof of Theorem \ref{thm: + result} for details). If we have two minimal $\Z$-components, say $X_0$ and $tX_0$, then $\alpha|_{\Z}$ on each component is free and so is $\alpha$, contradicting to our assumption.

Now, as the action is minimal and $X$ is infinite, every stabilizer subgroup must be of infinite index in $D_{\infty}$. 
By Lemma \ref{lem: subgroups of the infinite dihedral group}, we deduce that for each $x\in X$, if the stabilizer group $Stab(x)$ is not trivial, then $Stab(x)=\{e, s^nt\}$ for some 
$n\in \Z$. To check topological freeness of $\alpha$, we only 
need to show that for each $n\in\Z$, the closed set $X_n:=\{x\in X: s^ntx=x\}$ is nowhere dense.

Note that $x\in X_n$ iff $Stab(x)=\{e, s^nt\}$. Now, assume that $\exists ~\emptyset\neq U\subsetneq X_n$, where $U$ is an open set. Then we may pick any $y\in X\setminus U$ and use $\Z\curvearrowright X$ is minimal to deduce that for each $N$, $\exists~ |m|>N$ such that $s^my\in U$. Hence, $\{e, s^nt\}=Stab(s^my)=s^mStab(y)s^{-m}$. Since $Stab(y)=\{e\}$ or $\{e, s^kt\}$ for some $k\in\Z$, we deduce that $2m=n-k$, which gives us a contradiction since $m$ can be arbitrarily large. 
\end{proof}

\subsection{Bi-Lipschitz bijections of $\Z$}

We need the following result in \cite{BS}.

\begin{thm}[Benjamini-Shamov]\label{thm: bi-Lipschitz bijection of Z}
Let $f:\Z\to\Z$ be a bi-Lipschitz bijection ($\Z$ is equipped with its usual metric, namely $\rho(x, y):=|x-y|$). Then either $\sup_{x\in \Z}|f(x)-x|<+\infty$ or $\sup_{x\in \Z}|f(x)+x|<+\infty$. More precisely, $f(x)=\pm x+\mbox{const}+r(x)$, where $\sup_{x\in \Z}|r(x)|<\infty$.
\end{thm}

\subsection{Gottschalk-Hedlund theorem}
Another ingredient we need is the classical Gottschalk-Hedlund theorem \cite[Theorem 14.11]{GH}. We record the following version as presented in \cite[P. 102, Theorem 2.9.4]{KH} and explain how to use it in our setting. Note that the proof of \cite[P. 102, Theorem 2.9.4]{KH} still works without the metrizability assumption on the space.

\begin{thm}[Gottschalk-Hedlund]\label{thm: Gottschalk-Hedlund}
Let $X$ be a compact Hausdorff space, $f: X\to X$ a minimal continuous map and $g: X\to \mathbb{R}$ continuous such that $\sup_{n\in\N}|\sum_{i=0}^ng\circ f^i(x_0)|<\infty$ for some $x_0\in X$. Then there is a continuous $\phi: X\to\mathbb{R}$ such that $\phi\circ f-\phi=g$.
\end{thm}
Now, let $f: X\to X$ be a minimal continuous action and $c: \Z\times X\to \mathbb{R}$ be a continuous cocycle for this action. Set $g(x)=c(1, x)$, then a calculation shows that $\sum_{i=0}^ng\circ f^i(x_0)=c(n+1, x_0)$.
Thus, Gottschalk-Hedlund theorem actually shows that if $c$ is a bounded cocycle, i.e. $\sup_{(n, x)\in \Z\times X}|c(n, x)|<\infty$, then it is a continuous coboundary map. 

\section{existence of cocycles that are  not coboundaries}\label{section: existence of non-coboundaries}
For the proof of Theorem \ref{thm: - result}, we need to construct free minimal actions such that they admit continuous cocycles which are not coboundaries with values into finite groups. We collect some results along this direction.

\begin{proposition}\label{prop: existence of non-coboundaries for weak mixing actions}
Let $\Z\curvearrowright X$ be a minimal weakly mixing and continuous action on a compact metric space $X$. Then for each non-trivial group homomorphism $\phi$ from $\Z$ to $F$, where $F$ is a finite abelian group, $\phi$ is not a coboundary.
\end{proposition}
\begin{proof}
Assume that $\phi(s)=L(sx)-L(x)$ for some continuous map $L: X\to F$ and all $s\in \Z$ and all $x\in X$. Then $L(sx)-L(x)=L(sy)-L(y)$ for all $x, y\in X$. Then $L'(x, y):=L(x)-L(y)$ is a continuous map which is constant on each orbit of any point in $X\times X$. By assumption, we know that $L'$ is constant since the diagonal action $\Z\curvearrowright X\times X$ admits a point with dense orbit. Thus $L'(x, y)=L'(x, x)=0$ for all $x, y\in X$, i.e. $L$ is a constant function on $X$, hence $\phi(s)=0$ for all $s\in\Z$, i.e. $\phi$ is a trivial homomorphism, a contradiction.
\end{proof}

For the next result, we need the following known characterization of coboundaries using essential values for continuous cocycles in \cite{LM}, which was adapted from Schmidt's notion for measurable cocycles in \cite{Sc}.

Let $K$ be a locally compact group with the unit element $e$. By $K_{\infty}$ we denote the one point compactification of $K$, i.e. $K_{\infty}=K\cup \{\infty\}$. We may extend the group operation form the group $K$ onto the set $K_{\infty}$ by $g\cdot \infty=\infty$ for all $g\in K_{\infty}$. One can check that the operation $K_{\infty}\times K_{\infty}\ni (g_1, g_2)\to g_1g_2\in K_{\infty}$ is continuous. 

\begin{definition}[Definition 3.1 in \cite{LM}]
Let $T: \Z\curvearrowright X$ be a continuous action on a compact space $X$. Let $c: \Z\times X\to K$ be a continuous cocycle. We say that $r\in K_{\infty}$ is an \emph{essential value of $\phi$} if for each nonempty open $U\subset X$ and each neighborhood $V$ of $r$ there exists $N\in \Z$ such that 
\[U\cap T^{-N}U\cap \{x\in X: c(N, x)\in V\}\neq \emptyset.\] 
\end{definition}
The set of all essential values of $c$ is denote by $E_{\infty}(c)$. Put also $E(c)=E_{\infty}(c)\cap K$.

\begin{proposition}\label{prop: E(f) trivial iff f is a coboundary}
Let $\Z\curvearrowright X$ be a minimal action and $K$ be a compact abelain group, e.g. a finite abelian group. Let $c: \Z\times X\to K$ be a continuous cocycle. Then $E(c)=\{e\}$ iff $c$ is a coboundary, where $e$ denotes the neutral element in $K$.
\end{proposition}
\begin{proof}
$\Rightarrow$: this is Lemma 3.1 (ii) in \cite{LM}. $\Leftarrow$: one can check this directly using the definition of $E(c)$.   
\end{proof}
The following result might be known to experts, since we do not find a reference, we include the proof.
\begin{proposition}\label{prop: existence of non-coboundaries for odometers}
Let $(\Z, X)$ be a minimal odometer action, say $X=\lim\limits_{\leftarrow}\Z/{n_i\Z}$, where $n_i\mid n_{j}$ for all $i\leq j$ and $\lim_{j\to\infty}n_j=\infty$. Let $p$ be a prime number. Then the following statements are equivalent.\\
(1) Every continuous map $f: X\to \Z/{p\Z}$ gives rise to a continuous coboundary.\\
(2) $\sup_iord(p, n_i)=\infty$, where $ord(p, n_i):=\max_{p^k\mid n_i}k$.
\end{proposition}

\begin{proof}
(1)$\Rightarrow$(2): By Proposition \ref{prop: E(f) trivial iff f is a coboundary}, we know that for all continuous $f: X\to \Z/{p\Z}$, $E(f)=\{0\}$, i.e. for all $0\neq c\in \Z/{p\Z}$, there exists some non-empty open set $U\subset X$ such that for all $n\in\Z$, we have $x\in U\cap T^{-n}U\Rightarrow f(x)+f(Tx)+\cdots+f(T^{n-1}x)\neq c$. 

Now, denote by $\pi_i: X\to \Z/{n_i\Z}$ the natural projection map, i.e. $\pi_i((x_i+n_i\Z)_i)=x_i+n_i\Z$. Assume that $\sup_iord(p, n_i)$ is bounded, let us take some $j$ such that $ord(p, n_j)=\sup_iord(p, n_i)$ and we may assume $n_j>1$, then we define a map $f: X\to \Z/{p\Z}$ as the composition $X\overset{\pi_j}{\to} \Z/{n_j\Z}\overset{f'}{\to} \Z/{p\Z}$, where $f'(i+n_j\Z)=\delta_{i, 1}\in \Z/{p\Z}$ for all $0\leq i\leq n_1-1$. Take $c=1\in \Z/{p\Z}$, then we find some $U\subset X$ as above. We may shrink $U$ if necessary to assume $U=\pi_k^{-1}(i_0+n_k\Z)$ for some $k>j$ and $i_0$. Now, take any $x=(x_i+n_i\Z)_i\in U$ and $n=n_k\ell$ for some integer $\ell$ such that $\frac{n_k}{n_j}\ell\equiv 1~mod~p$ (This is possible since $p\nmid \frac{n_k}{n_j}$ by our choice of $j$ and $p$ is a prime number). Note that $n_k\mid n$ implies that $T^nx\in U$. Now, a calculation shows that
\begin{align*}
f(x)+f(Tx)+\cdots+f(T^{n-1}x)&=\sum_{s=0}^{n-1}f'(x_j+s+n_j\Z)\\
&=1\cdot (\frac{n}{n_j})=\frac{n_k\ell}{n_j}=1=c\in \Z/{p\Z}.
\end{align*}
This is a contradiction.

(2)$\Rightarrow$(1): the proof is similar as above. By Proposition \ref{prop: E(f) trivial iff f is a coboundary}, we aim to show that for all continuous $f: X\to \Z/{p\Z}$ and all $0\neq c\in \Z/{p\Z}$, there exists some non-empty set $U\subset X$ such that for all $n\in \Z$,  $x\in U\cap T^{-n}U\Rightarrow f(x)+f(Tx)+\cdots+f(T^{n-1}x)\neq c$. 

Pick $j$ large enough such that $f$ factors through $\pi_j: X\to \Z/{n_j\Z}$, i.e. $f=\pi_j\circ f'$ for some map $f': \Z/{n_j\Z}\to \Z/{p\Z}$. Since $\sup_iord(p, n_i)=\infty$, we may find $k>j$ such that $ord(p, n_k)>ord(p, n_j)$. Then let $U=\pi_k^{-1}(0+n_k\Z)$. Now, for each $n\in \Z$, if $x\in U\cap T^{-n}U$, then $x_k=0=x_k+n$ in $\Z/{n_k\Z}$, so $n_k\mid n$. Now, a calculation shows that
\begin{align*}
&f(x)+f(Tx)+\cdots+f(T^{n-1}x)\\
&=\sum_{s=0}^{n-1}f'(x_j+s+n_j\Z)\\
&=(\sum_{s=0}^{n_j-1}f'(s+n_j\Z))\frac{n}{n_j}\\
&=[(\sum_{s=0}^{n_j-1}f'(s+n_j\Z))\frac{n}{n_k}]\frac{n_k}{n_j}\in p\Z~\mbox{(since $ord(p, n_k)>ord(p, n_j)$)}.
\end{align*}
Thus, $0=f(x)+f(Tx)+\cdots+f(T^{n-1}x)\neq c$ in $\Z/{p\Z}$.
\end{proof}
From the above proposition, we can deduce that for the free minimal odometer action $\Z\curvearrowright \lim\limits_{\leftarrow}\frac{\Z}{2^i\Z}:=X$, it admits a continuous $\Z/{3\Z}$-valued cocycle that is not a coboundary.
\section{proofs}\label{section: proofs}

\subsection{Non-rigidity part}\label{subsection:-part}
Theorem \ref{thm: - result} is a direct corollary of the following proposition.
\begin{proposition}\label{prop: coe but not conjugacy}
Let $\alpha: \Z\curvearrowright X$ be a minimal free action on a compact Hausdorff space. Let $c, c': \Z\times X\to F$ be two continuous cocycles (w.r.t. the action $\alpha$) into a finite non-trivial group $F$. Consider the associated skew product actions $\widetilde{\alpha}: F\times \Z\curvearrowright X\times_cF$ and  $\widetilde{\alpha}': F\times \Z\curvearrowright X\times_{c'}F$ as defined in Subsection \ref{subsection: definition of skew product actions}. Then the following hold:\\
(i) Both $\widetilde{\alpha}$ and $\widetilde{\alpha}'$ are minimal free actions.\\
(ii) $\widetilde{\alpha}\overset{coe}{\sim}\widetilde{\alpha}'$.\\
(iii) If $c'\equiv e\in F$, the neutral element in $F$, then $\widetilde{\alpha}\overset{conj.}{\sim}\widetilde{\alpha}'$ implies that $c$ is cohomologous to a group homomorphism from $\Z$ into $C(F)$.\\
(iv) For each $F$ with trivial center, there exist $\alpha$, $c$ and $c'$ such that $\widetilde{\alpha}\overset{conj.}{\not\sim}\widetilde{\alpha}'$.
\end{proposition}
\begin{proof}
(i) We leave it as an exercise.\\
(ii) Recall that $\widetilde{\alpha}: F\times \Z\curvearrowright X\times_cF$ is defined as follows:
\[\widetilde{\alpha}_{(f, n)}(x, f')=(\alpha_n(x), c(n, x)f'f^{-1})~\mbox{for all $f$, $f'\in F$, $n\in \Z$ and $x\in X$}.\]
Define $\theta: (F\times \Z)\times (X\times_cF)\to F\times \Z$ by setting
\[\theta((t, n), (x, f))=(tf^{-1}c(n, x)^{-1}c'(n, x)f, n)~\mbox{for all $n\in \Z$, $t\in F$ and $x\in X$}.\]
One can check it is a continuous cocycle w.r.t. the skew product action $\widetilde{\alpha}$. Indeed, 
\begin{align*}
\theta((t_1, n_1)(t_2, n_2), (x, f))
&=\theta((t_1t_2, n_1+n_2), (x, f))\\
&=(t_1t_2f^{-1}c(n_1+n_2, x)^{-1}c'(n_1+n_2, x)f, n_1+n_2)
\end{align*}
\begin{align*}
&\theta((t_1, n_1), \widetilde{\alpha}_{(t_2, n_2)}(x, f))\theta((t_2, n_2), (x, f))\\
&=\theta((t_1, n_1), (\alpha_{n_2}(x), c(n_2, x)ft_2^{-1}))\theta((t_2, n_2), (x, f))\\
&=(t_1t_2f^{-1}c(n_2, x)^{-1}c(n_1, \alpha_{n_2}(x))^{-1}c'(n_1, \alpha_{n_2}(x))c(n_2, x)ft_2^{-1},n_1)\cdot\\&\quad\quad(t_2f^{-1}c(n_2, x)^{-1}c'(n_2, x)f,n_2)\\
&=(t_1t_2f^{-1}c(n_2, x)^{-1}c(n_1, \alpha_{n_2}(x))^{-1}c'(n_1, \alpha_{n_2}(x))c'(n_2, x)f, n_1+n_2)\\
&=(t_1t_2f^{-1}c(n_1+n_2, x)^{-1}c'(n_1+n_2, x)f,n_1+n_2).
\end{align*}
Thus, 
\begin{align*}
\theta((t_1, n_1)(t_2, n_2), (x, f))=\theta((t_1, n_1), \widetilde{\alpha}_{(t_2, n_2)}(x, f))\theta((t_2, n_2), (x, f)).
\end{align*}
Moreover, one can verify that
\[\widetilde{\alpha}_{(t,n)}(x, f)=\widetilde{\alpha}'_{\theta((t, n), (x, f))}(x, f).\]

By symmetry, one can also define the inverse cocycle for $\theta$, hence $\widetilde{\alpha}\overset{coe}{\sim}\widetilde{\alpha}'$.

(iii) First, we observe that for each $\Phi\in Aut(F\times \Z)$, we have $\Phi(t, 1)=(\epsilon(t)g,\pm 1)$ for some $\epsilon\in Aut(F)$ and $g\in C(F)$, the center of $F$. Indeed, since $F$ is finite, $\Phi(F\times \{0\})=F\times \{0\}$, thus we define $\epsilon=\Phi|_{F\times \{0\}}$. Next, we may write $\Phi(e, 1)=(g, \pm 1)$ for some $g\in F$. From $\Phi((e, 1)(t, 0))=\Phi((t, 0)(e, 1))$ for all $t\in F$, we deduce that $g\in C(F)$. Note that $\Phi(t,n)=(\epsilon(t)g^n, \pm n)$ for all $n\in\Z$ and $t\in F$. For such $\Phi$, we simply write $\Phi=(\epsilon(\cdot)g, \pm)$.

Next, assume $\widetilde{\alpha}$ is conjugate to $\widetilde{\alpha}'$, say via a homeomorphism $\tau: X\times_cF\to X\times_{c'}F$ and a group isomorphism $(\epsilon(\cdot)g, \pm): F\times \Z\cong F\times \Z$ defined in the last paragraph. Then we may write
\[\tau((x, f'))=(\phi(x,f'),\psi(x, f'))\] for some continuous maps $\phi: X\times_cF\to X$ and $\psi: X\times_c F\to F$.

Fix any $(x, f')\in X\times_cF$ and $(f, n)\in F\times \Z$, the $F\times \Z$-equivariance of $\tau$ shows that for all $x\in X$, $f$, $f'\in F$ and all $n\in \Z$, we have
\begin{align}\label{eq: test conjugacy}
\begin{split}
\alpha_{\pm n}(\phi(x, f'))&=\phi(\alpha_n(x), c(n, x)f'f^{-1})\\
c'(\pm n, \phi(x, f'))\psi(x, f')g^{-n}\epsilon(f)^{-1}&=\psi(\alpha_n(x), c(n, x)f'f^{-1}).
\end{split}
\end{align} 
Since $f$ is arbitrary in the 1st expression in \eqref{eq: test conjugacy}, we deduce that $\phi$ only depends on its 1st coordinate, so we may directly write $\phi(x,-)=\phi(x)$ and hence $\alpha_{\pm n}(\phi(x))=\phi(\alpha_n(x))$. Using $\tau$ is a homeomorphism, we can check that $\phi|_X$ is also a homeomorphism.

Now, set $f=c(n, x)$ and $f'=e$ into the 2nd line in \eqref{eq: test conjugacy}, we get that 
\[c'(\pm n, \phi(x))\psi(x, e)g^{-n}\epsilon(c(n, x))^{-1}=\psi(\alpha_n(x), e).\] Thus,
\begin{align}\label{eq: middle formula}
c'(\pm n, \phi(x))=\psi(\alpha_n(x), e)\epsilon(c(n, x))g^n\psi(x, e)^{-1}.
\end{align}

Since $\alpha_{\pm n}(\phi(x))=\phi(\alpha_n(x))$,  \eqref{eq: middle formula} may be rewritten as 
\begin{align*}
c'(\pm n, x)=\psi(\phi^{-1}(\alpha_{\pm n}(x)), e)\epsilon(c(n, \phi^{-1}(x)))g^n\psi(\phi^{-1}(x), e)^{-1}.
\end{align*} or simply
\begin{align}\label{eq: final task}
c'(n, x)=\psi'(\alpha_n(x))\epsilon(c(\pm n, \phi^{-1}(x)))g^{\pm n}\psi'(x)^{-1}~\mbox{for all $n\in \Z$ and $x\in X$,}
\end{align}
where $\psi'(x):=\psi(\phi^{-1}(x), e)$.

Now, let $c\equiv e\in F$, the neutral element in $F$, then \eqref{eq: final task} is the same as saying $c'$ is cohomologous to a group homomorphism $\Z\ni n\mapsto g^{\pm n}\in C(F)$.

(iv) Set $c\equiv e$. We are left to construct $\alpha$  and $c$ such that $\widetilde{\alpha}\overset{conj.}{\not\sim}\widetilde{\alpha}'$. It suffices to make sure the above identity \eqref{eq: final task} fails. 

Since $C(F)$ is trivial, we deduce $g=e$. It suffices to find some $\alpha$ and a cocycle $c': \Z\times X\to F$ (w.r.t. $\alpha$) which is not a coboundary. 

By Cauchy theorem, we may find some non-trivial $t\in F$ such that $\Z/{p\Z}\cong \langle t\rangle\subset F$ for some prime number $p$. By Lemma \ref{lem: independent of target group for being coboundaries}, it suffices to find some $\alpha$ and a continuous cocycle $c': \Z\times X\to \Z/{p\Z}$ which is not a coboundary.

We can apply 
Proposition \ref{prop: existence of non-coboundaries for weak mixing actions} or Proposition \ref{prop: existence of non-coboundaries for odometers} to construct a suitable free minimal and weakly mixing or free minimal odometer action $\Z\curvearrowright X$ and a continuous cocycle $c$ which is not a coboundary. 
\end{proof}

\subsection{Rigidity part}\label{subsection: +part}
We are ready to prove Theorem \ref{thm: + result}.
\begin{proof}[Proof of Theorem \ref{thm: + result}]

Let $D_{\infty}=\Z\rtimes \frac{\Z}{2\Z}=\langle s\rangle\rtimes \langle t\rangle$. Denote by $S$ the symmetric generating set $\{s^{\pm}, t^{\pm}\}$ for $D_{\infty}$. Denote by $|\cdot|_S$ the word length on $G$ with respect to $S$. Note that $|s^nt|_S=n+1$ and $|s^n|_S=n$. Let $d(\cdot,\cdot)$ be the right-invariant word metric on $D_{\infty}$ with respect to $S$, i.e. $d(g_1, g_2):=|g_2g_1^{-1}|_S$.

Let $c: D_{\infty}\times X\to D_{\infty}$ be the continuous orbit cocycle associated to a given continuous orbit equivalence.

\textbf{Claim 1:} for all $x\in X$, either $\sup_{n\in \Z}d(c(s^n, x), s^n)<\infty$ or $\sup_{n\in \Z}d(c(s^n, x), s^{-n})<\infty$.

\begin{proof}[Proof of Claim 1]
Define $\pi: \Z\to D_{\infty}$ by setting $\pi(s^{2n})=s^n$ and $\pi(s^{2n+1})=ts^n$ for all $n\in \Z$. One can verify that $\pi$ is a bi-Lipschitz bijection. In fact, $\frac{|n-m|}{2}\leq d(\pi(s^n),\pi(s^m))\leq 2|n-m|$ for all $n$, $m\in \Z$. 

Fix any $x\in X$, write $\phi_x(g)=c(g, x)$ for all $g\in D_{\infty}$. Define $\phi: \Z\to \Z$ by setting $\phi=\pi^{-1}\circ \phi_x\circ \pi$. Clearly, $\phi$ is again a bi-Lipschitz bijection by Proposition \ref{prop: from coe to bi-Lipschitz maps}. 

From Theorem \ref{thm: bi-Lipschitz bijection of Z}, we deduce that either $\sup_{n\in\Z}|\phi(n)-n|<\infty$ or $\sup_{n\in\Z}|\phi(n)+n|<\infty$. This implies Claim 1 holds.
\end{proof}

Define $X_+=\{x\in X: \sup_{n\in\Z}d(c(s^n, x), s^n)<\infty\}$ and $X_-=\{x\in X: \sup_{n\in\Z}d(c(s^n, x), s^{-n})<\infty\}$. Clearly, $X_+\cap X_-=\emptyset$ since $s$ has infinite order. From Claim 1, we know that $X=X_+\sqcup X_-$.

\textbf{Claim 2}: Both $X_+$ and $X_-$ are clopen subsets of $X$.

\begin{proof}[Proof of Claim 2]
We follow the idea while dealing with $\Z$-actions in \cite{BT}. For each $r\geq 1$, define $B(r):=\{ts^i, s^i:~|i|\leq r\}\subset D_{\infty}$. Note that $B(r)\nearrow D_{\infty}$ as $r\to\infty$.

First, for all $N>0$, there exists $K>0$ such that $B(N)\subset c(B(K), x)$ for all $x\in X$. Indeed, since $D_{\infty}\ni g\mapsto c(g, x)\in D_{\infty}$ is a bijection, we may find $K_x>0$ such that $B(N)\subset c(B(K_x), x)$ for all $x\in X$. Moreover, by continuity of $c$, we may find a small open neighborhood $V_x\ni x$ such that $B(N)\subset c(B(K_x), y)$ for all $y\in V_x$. Since $X=\cup_{x\in X}V_x$, compactness of $X$ implies we may find a finite subcover $X=\cup_{i=1}^nV_{x_i}$, set $K=\max_{1\leq i\leq n}K_{x_i}$. 

Let $N:=\max_{x\in X,g\in S}|c(g, x)|$, and $K$ as above for this $N$. Then for all $x\in X$,
either 
\[c(s^{>K}, x)\subseteq \{s^{>0}, ts^{>0}\}~\mbox{and}~c(s^{<-K}, x)\subseteq \{s^{<0}, ts^{<0}\}~\mbox{or}\]
\[c(s^{>K}, x)\subseteq \{s^{<0}, ts^{<0}\}~\mbox{and}~c(s^{<-K}, x)\subseteq \{s^{>0}, ts^{>0}\}.\] Here, $c(s^{>K}, x):=\{c(s^i, x): i>K\}$, and $c(s^{-K}, x)$ is similarly defined; $\{s^{>0}, ts^{>0}\}:=\{s^i, ts^j: i>0,~ j>0\}$ and $\{s^{<0}, ts^{<0}\}$ is similarly defined.

To see this, for each $|k|>K$, we have $c(s^k, x)\not\in B(N)$ as $c(-,x)$ is bijective. Since $|c(s^{\pm 1}, s^kx)|\leq N$, from
$c(s^{\pm 1+k}, x)=c(s^{\pm 1}, s^kx)c(s^k, x)$, we deduce that $c(s^{\pm 1+k},x)$ and $c(s^k, x)$ must have the same sign for the $s$-exponent when writing them as elements in $t\Z=t\langle  s\rangle$ or $\Z=\langle  s\rangle$. Then apply Claim 1 to see that if $c(s^{>K}, x)\subseteq \{s^{>0}, ts^{>0}\}$ holds, then $x\in X_+$ and thus $c(s^{-K}, x)\subseteq \{s^{<0}, ts^{<0}\}$ is automatic. 

Now, denote by $X_+':=\{x\in X: c(s^{>K}, x)\subseteq \{s^{>0}, ts^{>0}\}~\mbox{and}~c(s^{<-K}, x)\subseteq \{s^{<0}, ts^{<0}\}\}$ and $X_-':=\{x\in X: c(s^{>K}, x)\subseteq \{s^{<0}, ts^{<0}\}~\mbox{and}~c(s^{<-K}, x)\subseteq\{s^{>0}, ts^{>0}\}\}$.
  
From above, we have shown $X_+'\sqcup X_-'=X=X_+\sqcup X_-$. It is easy to see that $X_+'=X_+$ and $X_-'=X_-$. Hence, to finish the proof, we just need to observe that both $X_+'$ and $X_-'$ are closed as $c$ is continuous and $D_{\infty}$ is countable and discrete.  
\end{proof}

Define $a: \Z\times X\to D_{\infty}$ by setting 
\begin{align*}
a(s^n, x)=\begin{cases}
c(s^n, x)s^{-n},&~\mbox{if $x\in X_+$}\\
c(s^n, x)s^n,&~\mbox{if $x\in X_-$}.
\end{cases}
\end{align*}

\textbf{Claim 3:} $a$ is a continuous cocycle taking finitely many values as $n$ changes.

\begin{proof}[Proof of Claim 3]
From Claim 1, we know $a$ takes only finitely many values. 

Fix any $x\in X_+$ and $n\in \Z$. We have 
\begin{align*}
a(s^{m+n}, x)=c(s^{m+n}, x)s^{-m-n}=c(s^m, s^nx)c(s^n, x)s^{-n}s^{-m}\\
=\begin{cases}
a(s^m, s^nx)s^ma(s^n, x)s^{-m},~&\mbox{if $s^nx\in X_+$}\\
a(s^m, s^nx)s^{-m}a(s^n, x)s^{-m},~&\mbox{if $s^nx\in X_-$}.
\end{cases}
\end{align*}

Case 1: $s^nx\in X_+$.

Since both $a(s^{m+n}, x)$ and $a(s^m, s^nx)$ lie in a finite set as $m$ changes, we deduce that $\{s^ma(s^n, x)s^{-m}: m\in \Z\}$ is a finite set. Thus, we have $a(s^n, x)\in \Z$ and thus $a(s^{m+n}, x)=a(s^m, s^nx)a(s^n, x)$.

Case 2: $s^nx\in X_-$.

Since both $a(s^{m+n}, x)$ and $a(s^m, s^nx)$ lie in a finite set as $m$ changes, we deduce that $\{s^{-m}a(s^n, x)s^{-m}: m\in \Z\}$ is a finite set. Thus, $a(s^n, x)\in \Z t$. Once again, this implies that $a(s^{m+n}, x)=a(s^m, s^nx)a(s^n, x)$.
To sum up, we have shown that for all $x\in X_+$, $a(s^{m+n}, x)=a(s^m, s^nx)a(s^n, x)$ holds for all $m$, $n\in \Z$.

Similarly, one can show this also holds for all $x\in X_-$. We include the details for completeness.

Fix any $x\in X_-$ and $n\in \Z$. We have 
\begin{align*}
a(s^{m+n}, x)=c(s^{m+n}, x)s^{m+n}=c(s^m, s^nx)c(s^n, x)s^{m+n}\\
=\begin{cases}
a(s^m, s^nx)s^ma(s^n, x)s^m,~&\mbox{if $s^nx\in X_+$}\\
a(s^m, s^nx)s^{-m}a(s^n, x)s^m,~&\mbox{if $s^nx\in X_-$}.
\end{cases}
\end{align*}

Case i: $s^nx\in X_+$.

Since $a(s^{m+n}, x)$ and $a(s^m, s^nx)$ take finitely many values as $m$ changes, we get that $\{s^ma(s^n, x)s^m: m\in \Z\}$ is a finite set. Therefore, $a(s^n, x)\in \Z t$, which implies that $a(s^{m+n}, x)=a(s^m, s^nx)a(s^n, x)$.

Case ii: $s^nx\in X_-$.

This case can be checked similarly.
\end{proof}

Define $D: X\to D_{\infty}$ by 
\begin{align*}
D(x)=\begin{cases}
t,&~\mbox{if $x\in X_-$}\\
e,&~\mbox{if $x\in X_+$.}
\end{cases}
\end{align*}
From Claim 2, we know $D$ is a continuous map.
Moreover, one can verify that $D(s^nx)a(s^n, x)D(x)^{-1}\in \Z$ for all $x\in X$ and $n\in \Z$. 

Indeed, one can check this by discussing four cases depending on $x\in X_{\pm}$ and $s^nx\in X_{\pm}$. We explain the proof for the case $x\in X_+$ and $s^nx\in X_-$ below, the other three cases can be checked similarly.

By assumption, there exists a finite subset $F\subset D_{\infty}$ such that $c(s^k, x)s^{-k}\in F\ni c(s^k, s^nx)s^k$ for all $k\in \Z$. From the cocycle identity $c(s^n, x)=c(s^k, s^nx)^{-1}c(s^{k+n}, x)$, we deduce that $c(s^n, x)=s^kg_ks^{k+n}$ for some $g_k\in F^{-1}F$ for all $k\in \Z$. Taking a sufficiently large $k$, we deduce that $g_k\in t\Z=t\langle  s\rangle$ and thus $c(s^n, x)\in t\langle  s\rangle$. Hence $D(s^nx)a(s^n, x)D(x)^{-1}=tc(s^n, x)s^{-n}\in \Z=\langle  s\rangle$.

Thus, the map $a': \Z\times X\to \Z\subset D_{\infty}$ defined by 
\[a'(s^n, x):=D(s^nx)a(s^n, x)D(x)^{-1}\]
is a continuous cocycle taking finitely many values in $\Z$.

Take a $\Z$-minimal component, say $X_0$, i.e. $X_0=\overline{Orb(\Z, x_0)}$ for some $x_0\in X$ is minimal w.r.t. the sub $\Z$-action. Then since $\Z\lhd D_{\infty}$, we know $tX_0=\overline{Orb(\Z, tx_0)}$ is also a minimal $\Z$-component. Moreover, $X_0\cup tX_0$ is a $G$-invariant closed subset, hence $X=X_0\cup tX_0$ by minimality assumption. Clearly, either $X_0\cap tX_0=\emptyset$ or $X_0=tX_0$.

At this stage, we need to split the proof into two cases.

\textbf{Case I:} $X_0=tX_0$, i.e. the sub $\Z$-action $\Z\curvearrowright X$ is still minimal.

We may think of $a'$ as a cocycle taking values in $\R$ via the natural inclusion $\Z\hookrightarrow \R$. As we assume $\Z\curvearrowright X$ is still minimal, we may apply Gottschalk-Hedlund theorem, i.e. Theorem \ref{thm: Gottschalk-Hedlund} and Lemma \ref{lem: independent of target group for being coboundaries} to deduce there exists some continuous map $L: X\to \Z$ such that $a'(s^n, x)=L(s^nx)^{-1}L(x)$. Thus, letting $L'(x)=L(x)D(x)$, we deduce that
\[a(s^n, x)=L'(s^nx)^{-1}L'(x).\] 
Therefore, 
\begin{align*}
c(s^n, x)&=\begin{cases}
L'(s^nx)^{-1}L'(x)s^n,~&\mbox{if $x\in X_+$}\\
L'(s^nx)^{-1}L'(x)s^{-n},~&\mbox{if $x\in X_-$}
\end{cases}\\
&=L'(s^nx)^{-1}s^nL'(x)~\mbox{for all $x\in X$ and $n\in \Z$.}
\end{align*}
Here, the last equality can be checked by observing that $L(x)\in \Z$ which is abelian and $t$ acts on $\Z$ as a reflection.

\textbf{Claim 4:} there exists some $k\in \Z$ such that $c(t, x)=L'(tx)^{-1}(s^kt)L'(x)$ for all $x\in X$.

\begin{proof}[Proof of Claim 4]
Notice that on the one hand, $c(s^{-n}, x)=L'(s^{-n}x)^{-1}s^{-n}L'(x)$; on the other hand,
\begin{align*}
c(s^{-n}, x)=c(ts^nt^{-1}, x)=c(t, s^nt^{-1}x)c(s^n, t^{-1}x)c(t^{-1}, x)\\
=c(t, s^nt^{-1}x)L'(s^nt^{-1}x)^{-1}s^nL'(t^{-1}x)c(t^{-1}, x).
\end{align*}
Thus, 
\begin{align}\label{eq: pass to normalizer}
s^{-n}(L'(x)c(t^{-1},x)^{-1}L'(t^{-1}x)^{-1})s^{-n}=L'(s^{-n}x)c(t, s^nt^{-1}x)L'(s^nt^{-1}x)^{-1}.
\end{align}
Since the right hand side of the above takes only finitely many values as $n$ changes, we deduce that $L'(x)c(t^{-1},x)^{-1}L'(t^{-1}x)^{-1}\in \Z t$. 

Write $L'(x)c(t^{-1},x)^{-1}L'(t^{-1}x)^{-1}=s^{k(x)}t$ for some $k(x)\in\Z$. Clearly, $X\ni x\mapsto k(x)\in \Z$ is continuous from definition. Next, from \eqref{eq: pass to normalizer}, we may deduce that $c(t^{-1}, x)^{-1}=L'(x)^{-1}s^{k(x)}tL'(t^{-1}x)$ and $c(t, s^nt^{-1}x)=L'(s^{-n}x)^{-1}s^{k(x)}tL'(s^nt^{-1}x)$.

Since $t^2=e$, the above implies that $k(x)=k(s^nx)$ for all $n\in \Z$ and $x\in X$. Since the sub $\Z$-action is minimal and $k(\cdot)$ is continuous, we deduce that $k(x)\equiv k$ for all $x\in X$. 
\end{proof}
To sum up, we have shown that 
\begin{align*}
c(s^n, x)=L'(s^nx)^{-1}s^nL'(x), ~c(t, x)=L'(tx)^{-1}s^ktL'(x) ~\mbox{for all} ~n\in \Z ~\mbox{and all}~ x\in X.
\end{align*}
 Clearly, this is equivalent to say $c$ is cohomologous to $\phi\in Aut(D_{\infty})$, where $\phi$ is given by $\phi(s)=s$ and $\phi(t)=s^kt$ as mentioned in Lemma \ref{lem: subgroups of the infinite dihedral group}.

Finally, we can apply Proposition 4.4 in \cite{Li} to conclude that the two actions are conjugate.

\textbf{Case II:} $X_0\cap tX_0=\emptyset$, i.e. $X=X_0\sqcup tX_0$ and hence $X_0$ is a clopen subset.

We apply Gottschalk-Hedlund theorem to the minimal subactions $\Z\curvearrowright X_0$ and $\Z\curvearrowright tX_0$ respectively. Argue similarly as in Case I, we find two continuous maps $L'': X_0\to D_{\infty}$ and $L': tX_0\to D_{\infty}$ such that
\begin{align}\label{eq: expression for c in case II}
c(s^n, x)=\begin{cases}
L''(s^nx)^{-1}s^nL''(x),~\mbox{for all $x\in X_0$ and $n\in \Z$},\\
L'(s^nx)^{-1}s^nL'(x),~\mbox{for all $x\in tX_0$ and $n\in \Z$}.
\end{cases}
\end{align}
Fix any $x\in X_0$, we compute as follows:
\begin{align*}
c(s^{-n}, x)&=L''(s^{-n}x)^{-1}s^{-n}L''(x)\\
c(s^{-n}, x)&=c(ts^nt^{-1}, x)\\
&=c(t, s^nt^{-1}x)c(s^n, t^{-1}x)c(t^{-1}, x)\\
&=c(t, s^nt^{-1}x)[L'(s^nt^{-1}x)^{-1}s^nL'(t^{-1}x)]c(t^{-1},x).
\end{align*}
Thus, we deduce 
\[L''(s^{-n}x)c(t, s^nt^{-1}x)L'(s^nt^{-1}x)^{-1}=s^{-n}[L''(x)c(t^{-1}, x)^{-1}L'(t^{-1}x)^{-1}]s^{-n}.\]
As $\{L''(s^{-n}x)c(t, s^nt^{-1}x)L'(s^nt^{-1}x)^{-1}: n\in\Z\}$ is a finite set, we deduce that 
\begin{align*}
L''(x)c(t^{-1}, x)^{-1}L'(t^{-1}x)^{-1}
=L''(s^{-n}x)c(t, s^nt^{-1}x)L'(s^nt^{-1}x)^{-1}\in \Z t~\mbox{for all $n\in \Z$}.
\end{align*}

Let us write 
\begin{align}\label{eq: eq1 for two transfer maps}
L''(x)c(t^{-1}, x)^{-1}L'(t^{-1}x)^{-1}=s^{k(x)}t
\end{align}
for some map $k: X_0\to \Z$. Clearly, $k(\cdot)$ is a continuous map. Moreover, since $X_0$ is $\Z=\langle  s\rangle$-invariant, from the above identity, we deduce that
\begin{align*}
L''(s^{-n}x)c(t^{-1}, s^{-n}x)^{-1}L'(t^{-1}s^{-n}x)^{-1}&=s^{k(s^{-n}x)}t\\
[L''(s^{-n}x)c(t, s^nt^{-1}x)L'(s^nt^{-1}x)^{-1}]^{-1}&=[s^{k(x)}t]^{-1}.
\end{align*}
Multiply the above two expressions and use $t^2=e$, we deduce that $e=s^{k(s^{-n}x)-k(x)}$, i.e. $k(s^{-n}x)=k(x)$ for all $n\in \Z$. Since $X_0$ is $\Z=\langle s\rangle$-minimal, we deduce that $k$ is constant, say $k(x)=k$ for all $x\in X_0$. Then, \eqref{eq: eq1 for two transfer maps} can be simplified to 
\begin{align*}
L'(tx)=L'(t^{-1}x)=t^{-1}s^{-k}L''(x)c(t^{-1}, x)^{-1}.
\end{align*} 
From \eqref{eq: expression for c in case II}, a calculation shows that for each $x\in X_0$,
\begin{align*}
c(s^n, tx)&=L'(s^ntx)^{-1}s^nL'(tx)=L'(ts^{-n}x)^{-1}s^nL'(tx)\\
&=[t^{-1}s^{-k}L''(s^{-n}x)c(t^{-1}, s^{-n}x)^{-1}]^{-1}s^n[t^{-1}s^{-k}L''(x)c(t^{-1}, x)^{-1}]\\
&=c(t^{-1}, s^{-n}x)L''(s^{-n}x)^{-1}s^{-n}L''(x)c(t^{-1}, x)^{-1}.
\end{align*}
In fact, using cocycle identity, we may compute the full expression for $c$. More precisely, for all $x\in X_0$, we have
\begin{align}\label{eq: final expression for c in case II}
\begin{split}
c(s^n, x)&=L''(s^nx)^{-1}s^nL''(x),\\
c(s^nt, x)&=c(t^{-1}, s^{-n}x)L''(s^{-n}x)^{-1}s^{-n}L''(x), \\
c(s^n, tx)&=c(t^{-1}, s^{-n}x)L''(s^{-n}x)^{-1}s^{-n}L''(x)c(t^{-1}, x)^{-1},\\
c(s^nt, tx)&=L''(s^nx)^{-1}s^nL''(x)c(t, tx).
\end{split}
\end{align}
Now, to avoid confusion, let us write $\beta$ for the 2nd action, i.e. $gx=\beta_{c(g, x)}(x)$ for all $g\in D_{\infty} $ and $x\in X$.
Define a map 
\begin{align*}
X=X_0\sqcup tX_0&\overset{\phi}{\longrightarrow}X\\
\phi(x)&=\beta_{L''(x)}(x)\\
\phi(tx)&=\beta_{t^{-1}L''(x)}(x)~\mbox{for all}~x\in X_0.
\end{align*} 
Clearly, $\phi$ is continuous as $X_0$ is clopen and both $c$ and $L'': X_0\to D_{\infty}$ are continuous.

\textbf{Claim 5:} $\phi$ is the desired conjugacy between the two actions, i.e. $\phi$ is a homeomorphism and $\phi(gx)=\beta_g(\phi(x))$ for all $g\in D_{\infty}$ and $x\in X$.
\begin{proof}[Proof of Claim 5]
The conjugacy identity can be checked by computation using $\beta_{c(g, x)}(x)=gx$ and the above expression for $c(g, x)$ in \eqref{eq: final expression for c in case II}. We are left to show that $\phi$ is a bijection.

(a) $\phi$ is injective.

First, we check $\phi|_{X_0}$ is injective. Indeed, take any $x, y\in X_0$, suppose $\phi(x)=\phi(y)$, that is, $\beta_{L''(x)}(x)=\beta_{L''(y)}(y)$, i.e. $\beta_{L''(y)^{-1}L''(x)}(x)=y$.

Since $D_{\infty}\ni g\mapsto c(g, x)\in D_{\infty}$ is a bijection by Proposition \ref{prop: from coe to bi-Lipschitz maps}, we may find some $g\in D_{\infty}$ such that $c(g, x)=L''(y)^{-1}L''(x)$. Thus, $y=\beta_{c(g, x)}(x)=gx$. Since $x,y\in X_0$ and $X_0\cap tX_0=\emptyset$, we deduce that $g\in \Z$, say $g=s^n$. Hence, $c(s^n, x)=L''(s^nx)^{-1}L''(x)$. In view of \eqref{eq: final expression for c in case II}, we deduce $s^n=e$, i.e. $x=y$.

Clearly, this also shows that $\phi|_{tX_0}$ is injective. 

We are left to show that $\phi(x)\neq \phi(ty)$ for any $x, y\in X_0$. 

Assume not, then $\beta_{L''(x)}(x)=\beta_{t^{-1}L''(y)}(y)$, i.e. $\beta_{L''(y)^{-1}tL''(x)}(x)=y$. We may find some $g\in D_{\infty}$ such that $c(g, x)=L''(y)^{-1}tL''(x)$. Hence, $y=\beta_{c(g, x)}(x)=gx$. So $g\in \Z$, say $g=s^n$. Thus, $c(s^n, x)=L''(s^nx)^{-1}tL''(x)$. In view of \eqref{eq: final expression for c in case II}, this implies $s^n=t$, a contradiction.

(b) $\phi$ is surjective. 

First, we observe that the 2nd action also has two distinct minimal $\Z$-components. Suppose not, we may apply the proof of Case I to the 2nd action to see the first action is conjugate to the 2nd one. Thus, these two actions must have the same number of minimal $\Z$-components (as each automorphism of $D_{\infty}$ must fix the subgroup $\Z=\langle s\rangle$ globally), but this contradicts our assumption in Case II.

Let us write the two minimal $\Z$-components of the 2nd action as $Y_0$ and $\beta_t(Y_0)$, i.e. $X=Y_0\sqcup \beta_t(Y_0)$.

Now, since $\phi(gx)=\beta_g(\phi(x))$ for all $g\in D_{\infty}$ and all $x\in X$, we know $\phi(X_0), \phi(tX_0)\in \{Y_0, \beta_t(Y_0)\}$. As $\phi$ is injective and $X_0\cap tX_0=\emptyset$, we deduce that $\{\phi(X_0), \phi(tX_0)\}=\{Y_0, \beta_t(Y_0)\}$, so $\phi$ is surjective.
\end{proof}
\end{proof}
\begin{remark}\label{remark: torsion free assumption is unnecessary for full shifts}
By \cite{CJ, Co}, we know that for every finitely generated one-ended group $G$, its full shifts $G\curvearrowright A^G$ for finite $A$ are continuous cocycle superrigid actions with respect to any countable target groups. It was mentioned in \cite[Corollary 5]{CJ} that this can be combined with \cite[Theorem 1.6]{Li} to deduce the full shifts as above are continuous orbit equivalence rigid actions if $G$ is further assumed to be torsion free and amenable. In fact, this further assumption is unnecessary. Indeed, we just observe that from continuous cocycle superrigidity, we deduce $c(g, x)=L(gx)^{-1}\phi(g)L(x)$. Then take $x$ to be any fixed point for this full shift action and use $g\mapsto c(g, x)$ is a bijection for all $x\in X$, we deduce that $g\mapsto \phi(g)$ is automatically a group isomorphism of $G$. Hence, we may just apply \cite[Proposition 4.4]{Li} instead to get the conclusion.
\end{remark}

\section{concluding remarks}\label{sec: remarks}
Here are several remarks on the theorems and possible generalizations.

(1) Let $F$ be any non-trivial finite abelian group. From the proof of Proposition \ref{prop: coe but not conjugacy}, we know that if $\Z\curvearrowright X$ is a minimal action admitting a continuous cocycle $c: \Z\times X\to F$ which is not cohomologous to a group homomorphism from $\Z$ to $C(F)$, the center of $F$, then we can construct two continuously orbit equivalent but not conjugate actions.

In fact, for each non-trivial finite abelian group $F$, such a minimal action $\Z\curvearrowright X$ and a cocycle as above do exist (by taking $\Z\curvearrowright X$ to be a well-chosen Toeplitz system \cite[Example 1.20]{Gl}, as shown to me by Prof. Lema\'{n}czyk). By combining this with Lemma \ref{lem: independent of target group for being coboundaries}, we can deduce that for each finite non-trivial group $F$, $F\times \Z$ admits two continuously orbit equivalent but not conjugate minimal actions. 

(2) The construction in the proof of item (ii) Proposition \ref{prop: coe but not conjugacy} can be modified to deal with a general semi-direct product group $G=F\rtimes_{\sigma} \Z$ for a finite group $F$. Since $Aut(F\rtimes_{\sigma}\Z)$ becomes more involved and the existence of minimal actions that admit skew cocycles that are not trivial  into finite groups is not clear to us, we do not study the full generality here.

(3) Let $G=D_{\infty}$. Then we take a $\Z$-minimal component, say $X_0$, i.e. $X_0=\overline{Orb(\Z, x_0)}$ for some $x_0\in X$ is minimal w.r.t. the sub $\Z$-action. Then since $\Z\lhd G$, we know $tX_0=\overline{Orb(\Z, tx_0)}$ is also a minimal $\Z$-component. Moreover, $X_0\cup tX_0$ is a $G$-invariant closed subset, hence $X=X_0\cup tX_0$ if we assume $G\curvearrowright X$ is minimal. 

We observe below that once we have two minimal $\Z$-components, then the $G$-action is an induced action.
\begin{proposition}
With the above notations and assumptions, if $X_0\cap tX_0=\emptyset$, then the action  $G\curvearrowright X=X_0\sqcup tX_0$ is conjugate to the induced action $G\curvearrowright G/{\Z}\times X_0$ associated to a cocycle $\delta: G\times G/{\Z}\to \Z$.
\end{proposition}
\begin{proof}
Indeed, consider the natural lift map $L: G/{\Z}\to G$ given by $L(\Z)=e$ and $L(t\Z)=t$. Then for the associated cocycle $\delta: G\times G/\Z\to \Z$ given by $\delta(g, g'\Z)=L(gg'\Z)^{-1}gL(g'\Z)$ satisfies that
\begin{align*}
\delta(s^n, \Z)=s^n,~\delta(s^n, t\Z)=s^{-n},~\delta(s^nt, \Z)=s^{-n}~\mbox{and}~\delta(s^nt, t\Z)=s^n~\mbox{for all $n\in \Z$.}
\end{align*}
Now, define $\psi: X\to G/{\Z}\times X_0$ by setting
\begin{align*}
\psi(x)=\begin{cases}
(\Z, x)& \mbox{if $x\in X_0$}\\
(t\Z, tx)& \mbox{if $x\in tX_0$.}
\end{cases}
\end{align*}
Clearly, $\psi$ is a bijection and continuous (as $X_0$ is clopen). One can check that $\psi(gx)=g\psi(x)$ for all $g\in G$ and $x\in X$. 

Hence, $\psi$ is a conjugacy between $G\curvearrowright X=X_0\sqcup tX_0$ and the induced action $G\curvearrowright G/{\Z}\times X_0$ for a free minimal action $\Z\curvearrowright X_0$ (i.e. the sub $\Z$-action) w.r.t. the above chosen cocycle $c$.
\end{proof}

Next, consider any two induced actions $G\curvearrowright G/{\Z}\times X_0$ from two minimal topologically free $\Z$-actions on $X_0$ defined above and let $\phi$ be a homeomorphism which induces a continuous orbit equivalence between these two actions. Suppose we know that $\phi=id$, then one can show the two actions are conjugate directly. This shows the main difficulty for dealing with Case II in the proof of Theorem \ref{thm: + result} is that we do not know, a priori, how the homeomorphism behaves on the two minimal $\Z$-components.

Before proving the above assertion, let us fix some notation for the induced actions. Fix any lift $L: G/{\Z}\to G$, i.e. $L(g\Z)\Z=g\Z$ for all $g\in G$. We may further assume $L(\Z)=e$ and $L(t\Z)=t$. Then let $\delta: G\times G/{\Z}\to \Z$ be the associated cocycle $\delta(g, s\Z)=L(gs\Z)^{-1}gL(\Z)$. Observe that $\delta(s, \Z)=L(s\Z)^{-1}sL(\Z)=s$ for all $s\in \Z$.

Fix a minimal topologically free action $\alpha: \Z\curvearrowright X_0$, then the induced action $\tilde{\alpha}: G\curvearrowright X=G/{\Z}\times X_0$ is defined as follows:
$\tilde{\alpha}_s(g\Z, x):=(sg\Z, \alpha_{\delta(s, g\Z)}(x))$ for all $s, g\in G$ and $x\in X_0$.

Now it suffices to show the following holds.

\begin{proposition}\label{prop: 2 minimal Z-components}
Let $\alpha,\beta: \Z\curvearrowright X_0$ be two minimal topologically free actions. Let $\tilde{\alpha},\tilde{\beta}: G\curvearrowright X:= G/{\Z}\times X_0$ be the associated induced actions as recalled above. Then the following hold.\\
(1) $\tilde{\alpha}\overset{coe}{\sim}\tilde{\beta}~\mbox{via the identity homeomorphism}\Rightarrow \alpha\overset{coe}{\sim}\beta$.\\
(2) $\alpha\overset{coe}{\sim}\beta\Rightarrow  \tilde{\alpha}\overset{conj.}{\sim}\tilde{\beta}$.
\end{proposition}
\begin{proof}
(1) Let $c: G\times X\to G$ be the continuous orbit cocycle w.r.t. $\tilde{\alpha}$, i.e. $\tilde{\alpha}_g(\tilde{x})=\tilde{\beta}_{c(g, \tilde{x})}(\tilde{x})$, where $\tilde{x}=(s\Z, x)\in X$ is any point. Then a calculation using the definition of the induced actions shows that
\begin{align*}
gs\Z=c(g, \tilde{x})s\Z,~~\alpha_{\delta(g, s\Z)}(x)=\beta_{\delta(c(g, \tilde{x}), s\Z)}(x),~\forall~g, s\in G,~\forall~x\in X_0.
\end{align*}
Notice that this implies that $c(g, \tilde{x})\Z=g\Z$ as $\Z\lhd G$. In particular, $c(s, \tilde{x})\in \Z$ for all $s\in \Z$.  

Now, we define the following map $\theta: \Z\times X_0\to \Z$ given by $\theta(s, x)=c(s, \tilde{x})$, where $\tilde{x}:=(\Z, x)$.

Clearly, $\theta$ is well-defined and continuous. Then, we verify that $\theta$ is a cocycle w.r.t. $\alpha$.

Take any $s_1, s_2\in \Z$, and $x\in X_0$, we have
\begin{align*}
\theta(s_1s_2, x)=c(s_1s_2, \tilde{x})=c(s_1, \tilde{\alpha}_{s_2}(\tilde{x}))c(s_2, \tilde{x})\\
\theta(s_1, \alpha_{s_2}(x))\theta(s_2, x)=c(s_1, \widetilde{\alpha_{s_2}(x)})c(s_2,\tilde{x}).
\end{align*}
It suffices to check that $\tilde{\alpha}_{s_2}(\tilde{x})=\widetilde{\alpha_{s_2}(x)}$.
Recall that $\tilde{x}=(\Z,x)$ and $\tilde{\alpha}_{s_2}(\tilde{x})=(s_2\Z, \alpha_{\delta(s_2,\Z)}x)=(\Z, \alpha_{s_2}(x))$. Meanwhile, $\widetilde{\alpha_{s_2}(x)}=(\Z,\alpha_{s_2}(x))$.

Now, let us check that $\theta$ gives us a coe between $\alpha$ and $\beta$.

Indeed, we have
\begin{align*}
\alpha_s(x)&=\alpha_{\delta(s, \Z)}(x)=\beta_{\delta(c(s, \tilde{x}), \Z)}(x)\\
&=\beta_{L(c(s, \tilde{x})\Z)^{-1}c(s, \tilde{x})L(\Z)}(x)\\
&=\beta_{L(s\Z)^{-1}\theta(s, x)}(x)\\
&=\beta_{\theta(s, x)}(x), ~\forall~s\in \Z,~\forall~x\in X_0.
\end{align*}
(2) Let $\theta: \Z\times X_0\to \Z$ be the cocycle from the coe. By \cite{BT}, we know that $\theta(s, x)=f(\alpha_s(x))^{-1}s^{\pm}f(x)$ for some continuous map $f: X\to \Z$. One can check $X_0\ni x\mapsto \beta_{f(x)}(x)\in X_0$ intertwines $\alpha_s$ with $\beta_{s^{\pm}}$, which implies $Y\ni (g\Z, x)\mapsto (g\Z, \beta_{f(x)}(x))\in X$ intertwines $\tilde{\alpha_g}$ with $\tilde{\beta}_{\tau(g)}$. Here, $\tau\in Aut(G)$ is defined as follows:
\begin{align*}
\tau=\begin{cases}
id,~\mbox{ if $\theta(s, x)=f(\alpha_s(x))^{-1}sf(x)$ for all $s\in\Z$}\\
\phi,~\mbox{ if $\theta(s, x)=f(\alpha_s(x))^{-1}s^{-1}f(x)$ for all $s\in\Z$}
\end{cases}
\end{align*}
Here, $\tau\in Aut(G)$ is determined by $\tau(s)=s^{-1}$ for all $s\in \Z$ and $\tau(t)=t$ for the reflection $t\in G$.
\end{proof}

\subsection*{Acknowledgements}

This work is partially supported by NSFC grant no. 12001081. He is grateful to Prof. Nhan-Phu Chung and Prof. Xin Li for useful discussion, to Prof. Xin Li for sharing his unpublished note on coe rigidity dating back to 2016. He also thanks Prof. Mariusz Lema\'{n}czyk and Prof. Zhengxing Lian for very helpful correspondence on cocycles. He also thanks the referee for his/her helpful suggestion which improves the presentation.

\end{document}